\documentclass[12pt]{article}
\usepackage{amsthm, amsmath, amssymb, amsfonts, latexsym}
\date{}

\title{On self-similar solutions of a multi-phase Stefan problem in the half-line}
\author{E.Yu. Panov \\ Yaroslav-the-Wise Novgorod State University, \\ Veliky Novgorod, Russia}

\theoremstyle{plain}
\newtheorem{theorem}{Theorem}[section]
\newtheorem{lemma}{Lemma}[section]
\newtheorem{proposition}{Proposition}[section]
\theoremstyle{definition}

\newtheorem{remark}{Remark}[section]
\newtheorem{example}{Example}[section]
\numberwithin{equation}{section}

\newcommand{\R}{{\mathbb R}}
\newcommand{\N}{{\mathbb N}}
\newcommand{\D}{{\mathcal D}}

\newcommand{\const}{\mathrm{const}}
\begin{document}
\maketitle

\begin{abstract}
We study self-similar solutions of a multi-phase Stefan problem for a heat equation on the half-line $x>0$ with a constant initial data and with Dirichlet or Neumann boundary conditions. In the case of Dirichlet boundary condition we prove that a nonlinear algebraic system for determination of the free boundaries is gradient one and the corresponding potential is an explicitly written strictly convex and coercive function. Therefore, there exists a unique minimum point of the potential, coordinates of this point determine free boundaries and provide the desired solution. This result is also extended to the case of infinitely many phase transitions. In the case of Neumann boundary condition we demonstrate that the problem may have solutions with different numbers (called types) of phase transitions. For each fixed type $n$ the system for determination of the free boundaries is again gradient and the corresponding potential is proved to be strictly convex and coercive, but in some wider non-physical domain. On the base of these properties it is proved that there exists a unique solution of Stefan-Neumann problem, and we also provide precise conditions to specify the type of the solution.
Bibliography$:$
$5$ titles.
\end{abstract}



\section{Stefan problem with Dirichlet boundary condition}\label{sec1}

In a quarter-plane $\Pi_+=\{ \ (t,x)\in\R^2 \ | \ t,x>0 \ \}$ we consider the multi-phase Stefan problem for the heat equation
\begin{equation}\label{1}
u_t=a_i^2u_{xx}, \quad u_i<u<u_{i+1}, \ i=0,\ldots,m,
\end{equation}
where $u_0\le u_1<\cdots<u_m<u_{m+1}=u_D$, $u_i$, $i=1,\ldots,m$, being the temperatures of phase transitions, $a_i>0$,
$i=0,\ldots,m$, are the diffusivity constants. We will study continuous piecewise smooth solutions $u=u(t,x)$ in $\Pi_+$ satisfying (\ref{1}) in the classical sense in the domains $u_i<u(t,x)<u_{i+1}$, $i=0,\ldots,m$, filled with the phases. On the unknown lines $x=x_i(t)$ of phase transitions where $u=u_i$ the following Stefan condition
\begin{equation}\label{St}
d_ix_i'(t)+k_iu_x(t,x_i(t)+)-k_{i-1}u_x(t,x_i(t)-)=0
\end{equation}
is postulated, where $k_i>0$ is the thermal conductivity of the $i$-th phase, while $d_i\ge 0$ is the Stefan number (the latent specific heat) for the $i$-th phase transition. In (\ref{St}) the unilateral limits $u_x(t,x_i(t)+)$, $u_x(t,x_i(t)-)$ on the line $x=x_i(t)$ are taken from the domain corresponding to the warmer/colder phase, respectively. For physical reasons,
the Stefan numbers $d_i$ should be positive. We will study an even more general case $d_i\ge 0$, assuming that $d_1>0$ if $u_0=u_1$.

In this case the problem (\ref{1}), (\ref{St}) is well-posed for $u_0\le u\le u_D$ and reduces to a degenerate nonlinear diffusion  equation (see \cite{Kam}, \cite[Chapter 5]{LSU})
\begin{equation}\label{diff}
\beta(u)_t-\alpha(u)_{xx}=0,
\end{equation}
where $\alpha(u)$, $\beta(u)$ are strictly increasing functions on $[u_0,u_D]$ linear on each interval $(u_i,u_{i+1})$,
$i=0,\ldots,m$, with slopes $\alpha'(u)=k_i$, $\beta'(u)=k_i/a_i^2$, and such that
\[
\alpha(u_i+)-\alpha(u_i-)=0, \quad \beta(u_i+)-\beta(u_i-)=d_i, \ i=1,\ldots,m
\]
(here we agree that $\alpha(u_1-)=\alpha(u_0)$, $\beta(u_1-)=\beta(u_0)$ in the case $u_1=u_0$).

We will study the initial-boundary value problem with constant initial and Dirichlet boundary data
\begin{equation}\label{2}
u(0,x)=u_0 \ \forall x>0, \quad u(t,0)=u_D \ \forall t>0.
\end{equation}
By the invariance of our problem under the transformation group
$(t,x)\to (\lambda^2 t, \lambda x)$, $\lambda\in\R$, $\lambda>0$, it is natural to seek a self-similar solution of problem (\ref{1}), (\ref{St}), (\ref{2}), which has the form $u(t,x)=u(\xi)$, $\xi=x/\sqrt{t}$. In view of (\ref{2}),
\[u(0)=u_D, \quad u(+\infty)\doteq\lim_{\xi\to+\infty} u(\xi)=u_0<u_D.\]
Thus, it is natural to suppose that the function $u(\xi)$ decreases. The case when $u_D\le u_0$ can be treated similarly.
Certainly, in this case the function $u(\xi)$ should be increasing.

For the heat equation
$u_t=a^2 u_{xx}$ a self-similar solution must satisfy the linear ODE $a^2u''=-\xi u'/2$, the general solution of which is
\[
u=C_1F(\xi/a)+C_2, \ C_1,C_2=\const, \mbox{ where } F(\xi)=\frac{1}{\sqrt{\pi}}\int_0^\xi e^{-s^2/4}ds.
\]
This allows to write our solution in the form
\begin{align}\label{3}
u(\xi)=u_i+\frac{u_{i+1}-u_i}{F(\xi_{i+1}/a_i)-F(\xi_i/a_i)}(F(\xi/a_i)-F(\xi_i/a_i)), \\ \nonumber
\xi_{i+1}<\xi<\xi_i, \ i=0,\ldots,m,
\end{align}
where $+\infty=\xi_0>\xi_1>\cdots>\xi_m>\xi_{m+1}=0$ and we agree that $\displaystyle F(+\infty)=\frac{1}{\sqrt{\pi}}\int_0^{+\infty} e^{-s^2/4}ds=1$.
Notice that the function $u(\xi)$ is constant on the interval $(\xi_1,+\infty)$ if $u_0=u_1$.
The parabolas $\xi=\xi_i$, $i=1,\ldots,m$, where $u=u_i$, are free boundaries, which must be determined by conditions (\ref{St}).
In the variable $\xi$ these conditions have the form (cf. \cite[Chapter XI]{CJ})
\begin{equation}\label{4}
d_i\xi_i/2+\frac{k_i(u_{i+1}-u_i)F'(\xi_i/a_i)}{a_i(F(\xi_{i+1}/a_i)-F(\xi_i/a_i))}-
\frac{k_{i-1}(u_i-u_{i-1})F'(\xi_i/a_{i-1})}{a_{i-1}(F(\xi_i/a_{i-1})-F(\xi_{i-1}/a_{i-1}))}=0,
\end{equation}
$i=1,\ldots,m$. Formally, we can consider solution (\ref{3}) also in the case $u_m=u_{m+1}$ when $u(\xi)\equiv u_D$ for
$0<\xi<\xi_m$ but in this case condition (\ref{4}) with $i=m$ reduces to the relation
\[
d_m\xi_m/2-\frac{k_{m-1}(u_m-u_{m-1})F'(\xi_m/a_{m-1})}{a_{m-1}(F(\xi_m/a_{m-1})-F(\xi_{m-1}/a_{m-1}))}=0,\]
which is impossible since the left-hand side of this relation is strictly positive. This is the reason why we exclude the case $u_m=u_{m+1}$ in our setting. On the contrary, the case $u_1=u_0$ is correct and does not cause any problems if $d_1>0$ because in this case (\ref{4}) with $i=1$ provides the relation
\[
d_1\xi_1/2+\frac{k_1(u_2-u_1)F'(\xi_1/a_1)}{a_1(F(\xi_2/a_1)-F(\xi_1/a_1))}=0,\]
containing the terms of opposite signs.

In the case $m=1$ system (\ref{4}) reduces to a single equation, which can be easily analysed. As a result, we obtain the classical Neumann solution of the Stefan problem.
To investigate the nonlinear system (\ref{4}) in general case of arbitrary number $m$ of phase transitions, we notice that it is a gradient one and coincides with the equality
$\nabla E(\bar\xi)=0$, where the function
\begin{align}\label{5}
E(\bar\xi)=-\sum_{i=0}^m  k_i(u_{i+1}-u_i)\ln (F(\xi_i/a_i)-F(\xi_{i+1}/a_i))+\sum_{i=1}^m d_i\xi_i^2/4, \\ \nonumber \bar\xi=(\xi_1,\ldots,\xi_m)\in\Omega,
\end{align}
the open convex domain $\Omega\subset\R^m$ is given by the inequalities $\xi_1>\cdots>\xi_m>0$.
Observe that $E(\bar\xi)\in C^\infty(\Omega)$. Since the function $F(\xi)$ takes values in the interval $(0,1)$, all the terms in expression (\ref{5}) are nonnegative while some of them are strictly positive. Therefore, $E(\bar\xi)>0$.

\subsection{Coercivity of the function $E$ and existence of a solution}

Let us introduce the sub-level sets \[\Omega(c)=\{ \ \bar\xi\in\Omega \ | \ E(\bar\xi)\le c \ \}, \quad c>0.\]

\begin{proposition}[coercivity]\label{th1}
The sets $\Omega(c)$ are compact for each $c>0$.
In particular, the function $E(\bar\xi)$ reaches its minimal value.
\end{proposition}

\begin{proof}
If $\bar\xi=(\xi_1,\ldots,\xi_m)\in\Omega(c)$ then
\begin{align}\label{co1a}
 -k_i(u_{i+1}-u_i)\ln (F(\xi_i/a_i)-F(\xi_{i+1}/a_i))\le E(\bar\xi)\le c, \quad i=0,\ldots,m; \\
\label{co1b}
d_i\xi_i^2/4\le E(\bar\xi)\le c, \quad i=1,\ldots,m.
\end{align}
It follows from (\ref{co1a}) with $i=m$ that $F(\xi_m/a_m)\ge e^{-c/(k_m(u_{m+1}-u_m))}$, which implies
the low bound
\[\xi_m\ge r_1=a_mF^{-1}(e^{-c/(k_m(u_{m+1}-u_m))}).\]
Similarly, we derive from (\ref{co1a}) with $i=0$ that in the case $u_1>u_0$
\[1-F(\xi_1/a_0)\ge e^{-c/(k_0(u_1-u_0))}\] (notice that $F(\xi_0/a_0)=F(+\infty)=1$). Therefore,
$F(\xi_1/a_0)\le 1-e^{-c/(k_0(u_1-u_0))}$. This implies the upper bound $\xi_1\le a_0F^{-1}(1-e^{-c/(k_0(u_1-u_0))})$.
On the other hand, if $u_0=u_1$ then $d_1>0$ and it follows from (\ref{co1b}) with $i=1$ that $\xi_1\le (4c/d_1)^{1/2}$. We summarize that in any case
\[
\xi_1\le r_2=\left\{\begin{array}{lcr}a_0F^{-1}(1-e^{-c/(k_0(u_1-u_0))}) & , & u_0<u_1, \\
(4c/d_1)^{1/2} & , & u_0=u_1. \end{array}\right.
\]
Further, it follows from (\ref{co1a}) that for all $i=1,\ldots,m-1$
\begin{equation}\label{6}
F(\xi_i/a_i)-F(\xi_{i+1}/a_i)\ge \delta_1\doteq\exp(-c/\alpha)>0,
\end{equation}
where $\displaystyle\alpha=\min_{i=1,\ldots,m-1}k_i(u_{i+1}-u_i)>0$. Since $F'(\xi)=\frac{1}{\sqrt{\pi}}e^{-\xi^2/4}<1$, the function $F(\xi)$ is Lipschitz with constant $1$, and it follows from (\ref{6}) that
\[
(\xi_i-\xi_{i+1})/a_i\ge F(\xi_i/a_i)-F(\xi_{i+1}/a_i)\ge \delta_1, \quad i=1,\ldots,m-1,
\]
and we obtain the estimates $\xi_i-\xi_{i+1}\ge\delta=\delta_1\min a_i$. Thus, the set $\Omega(c)$ is contained in a compact
\[
K=\{ \ \bar\xi=(\xi_1,\ldots,\xi_m)\in\R^m \ | \ r_2\ge\xi_1\ge\cdots\ge\xi_m\ge r_1, \ \xi_i-\xi_{i+1}\ge\delta \ \forall i=1 ,\ldots,m-1 \ \}.
\]
Since $E(\bar\xi)$ is continuous on $K$, the set $\Omega(c)$ is a closed subset of $K$ and therefore is compact. For $c>N\doteq\inf E(\bar\xi)$, this set is not empty and the function $E(\bar\xi)$ reaches on it a minimal value, which is evidently equal to $N$.
\end{proof}

We have established the existence of minimal value $E(\bar\xi_0)=\min E(\bar\xi)$. At the point $\bar\xi_0$ the required condition $\nabla E(\bar\xi_0)=0$ is satisfied, and $\bar\xi_0$ is a solution of the system (\ref{4}). The coordinates
of $\bar\xi_0$ determine the solution (\ref{3}) of our Stefan problem. Thus, we have established the following existence result.

\begin{theorem}\label{th1a}
There exists a self-similar solution (\ref{3}) of the problem (\ref{1}), (\ref{St}), (\ref{2}).
\end{theorem}

\subsection{Convexity of the function $E$ and uniqueness of a solution}

In this section we prove that the function $E(\bar\xi)$ is strictly convex. Since a strictly convex function can have at most one critical point (and it is necessarily a global minimum), the system (\ref{4}) has at most one solution, that is, a self-similar solution (\ref{3}) of the problem (\ref{1}), (\ref{St}), (\ref{2}) is unique. We will need the following simple lemma proven in \cite{Pan1} (see also \cite{Pan2}). For the sake of completeness we provide it with the proof.

\begin{lemma}\label{lem1}
The function $P(x,y)=-\ln (F(x)-F(y))$ is strictly convex in the half-plane $x>y$.
\end{lemma}

\begin{proof}
The function $P(x,y)$ is infinitely differentiable in the domain $x>y$. To prove the lemma, we need to establish that the Hessian $D^2 P$ is positive definite at every point. By the direct computation we find
\begin{align*}
\frac{\partial^2}{\partial x^2} P(x,y)=\frac{(F'(x))^2-F''(x)(F(x)-F(y))}{(F(x)-F(y))^2}, \\
\frac{\partial^2}{\partial y^2} P(x,y)=\frac{(F'(y))^2-F''(y)(F(y)-F(x))}{(F(x)-F(y))^2}, \
\frac{\partial^2}{\partial x\partial y} P(x,y)=-\frac{F'(x)F'(y)}{(F(x)-F(y))^2}.
\end{align*}
We have to prove positive definiteness of the matrix $Q=(F(x)-F(y))^2 D^2 P(x,y)$ with the components
\begin{align*}
Q_{11}=(F'(x))^2-F''(x)(F(x)-F(y)), \\ Q_{22}=(F'(y))^2-F''(y)(F(y)-F(x)), \ Q_{12}=Q_{21}=-F'(x)F'(y).
\end{align*}
Since $F'(x)=\frac{1}{\sqrt{\pi}}e^{-x^2/4}$, then $F''(x)=-\frac{x}{2}F'(x)$ and the diagonal elements of this matrix can be written in the form
\begin{align*}
Q_{11}=F'(x)(\frac{x}{2}(F(x)-F(y))+F'(x))= \\ F'(x)(\frac{x}{2}(F(x)-F(y))+(F'(x)-F'(y)))+F'(x)F'(y), \\
Q_{22}=F'(y)(\frac{y}{2}(F(y)-F(x))+(F'(y)-F'(x)))+F'(x)F'(y).
\end{align*}
By Cauchy mean value theorem there exists such a value $z\in (y,x)$ that
\[
\frac{F'(x)-F'(y)}{F(x)-F(y)}=\frac{F''(z)}{F'(z)}=-z/2.
\]
Therefore,
\begin{align*}
Q_{11}=F'(x)(F(x)-F(y))(x-z)/2+F'(x)F'(y), \\ Q_{22}=F'(y)(F(x)-F(y))(z-y)/2+F'(x)F'(y),
\end{align*}
and it follows that $Q=R_1+F'(x)F'(y)R_2$, where $R_1$ is a diagonal matrix with the positive diagonal elements
$F'(x)(F(x)-F(y))(x-z)/2$, $F'(y)(F(x)-F(y))(z-y)/2$ while $R_2=\left(\begin{smallmatrix} 1 & -1 \\ -1 & 1\end{smallmatrix}\right)$. Since $R_1>0$, $R_2\ge 0$, then the matrix $Q>0$, as was to be proved.
\end{proof}

\begin{remark}\label{rem1}
In addition to Lemma~\ref{lem1} we observe that the functions $P(x,0)=-\ln F(x)$, $P(+\infty,x)=-\ln(1-F(x))$
of single variable $x$ are strictly convex on $(0,+\infty)$. In fact, it follows from Lemma~\ref{lem1} in the limit as $y\to 0$ that the function $P(x,0)$ is convex on $(0,+\infty)$, moreover,
\[
(F(x))^2\frac{d^2}{dx^2}P(x,0)=F'(x)(\frac{x}{2}F(x)+F'(x))=\lim_{y\to 0}Q_{11}\ge 0.
\]
Since $F'(x)>0$, we find, in particular, that $\frac{x}{2}F(x)+F'(x)\ge 0$.
If $\frac{d^2}{dx^2}P(x,0)=0$ at some point $x=x_0$ then $0=\frac{x_0}{2}F(x_0)+F'(x_0)$ is the minimum of the nonnegative function $\frac{x}{2}F(x)+F'(x)$. Therefore, its derivative $(\frac{x}{2}F+F')'(x_0)=0$. Since $F''(x)=-\frac{x}{2}F'(x)$, this derivative
\[
(\frac{x}{2}F+F')'(x_0)=F(x_0)/2+\frac{x_0}{2}F'(x_0)+F''(x_0)=F(x_0)/2>0.
\]
But this contradicts our assumption. We conclude that $\frac{d^2}{dx^2}P(x,0)>0$ and the function $P(x,0)$ is strictly convex.

The strong convexity of the function $P(+\infty,x)=-\ln(1-F(x))$ is proved similarly. For the sake of completeness, we
provide the details.
In the limit as $x<y\to+\infty$ we derive from Lemma~\ref{lem1} that the function $P(+\infty,x)=\lim\limits_{y\to+\infty}P(y,x)$ is convex on $\R$ and
\[(1-F(x))^2\frac{d^2}{dx^2}P(+\infty,x)=F'(x)(\frac{x}{2}(F(x)-1)+F'(x))\ge 0.\]
If $\displaystyle\frac{d^2}{dx^2}P(+\infty,x)=0$ at some point $x=x_0\in\R$ then $x_0$ is a minimum point of the nonnegative function
$\frac{x}{2}(F(x)-1)+F'(x)$. Therefore,
\[
0=(\frac{x}{2}(F(x)-1)+F'(x))'(x_0)=(F(x_0)-1)/2+F''(x_0)+F'(x_0)x_0/2=(F(x_0)-1)/2<0.
\]
This contradiction implies that $\displaystyle\frac{d^2}{dx^2}P(+\infty,x)>0$ for all $x\in\R$ and, therefore, the function $P(+\infty,x)$ is strictly convex (even on the whole line $\R$).
\end{remark}

Now we are ready to prove the expected convexity of $E(\bar\xi)$.

\begin{proposition}\label{th2}
The function $E(\bar\xi)$ is strictly convex on $\Omega$.
\end{proposition}

\begin{proof}
We introduce the functions
\[E_i(\bar\xi)=-k_i(u_{i+1}-u_i)\ln (F(\xi_i/a_i)-F(\xi_{i+1}/a_i)), \quad i=0,\ldots,m.\]
By Lemma~\ref{lem1} and Remark~\ref{rem1} all these functions are convex.
Since
\[
E(\bar\xi)=\sum_{i=1}^m E_i(\bar\xi)+ E_0(\bar\xi)+\sum_{i=1}^m d_i\xi_i^2/4
\]
and all functions in this sum are convex, it is sufficient to prove strong convexity of the sum
\[
\tilde E(\bar\xi)=\sum_{i=1}^m E_i(\bar\xi).
\]
By Lemma~\ref{lem1} and Remark~\ref{rem1} all terms in this sum are convex functions. Therefore, the function $\tilde E$ is convex as well. To prove the strict convexity, we assume that for some vector $\zeta=(\zeta_1,\ldots,\zeta_m)\in\R^m$.
\begin{equation}\label{deg}
D^2 \tilde E(\bar\xi)\zeta\cdot\zeta=\sum_{i,j=1}^m \frac{\partial^2 \tilde E(\bar\xi)}{\partial\xi_i\partial\xi_j}\zeta_i\zeta_j=0
\end{equation}
Since
\[
0=D^2 \tilde E(\bar\xi)\zeta\cdot\zeta=\sum_{i=1}^m D^2 E_i(\bar\xi)\zeta\cdot\zeta
\]
while all the terms are nonnegative, we conclude that
\begin{equation}\label{deg1}
D^2 E_i(\bar\xi)\zeta\cdot\zeta=0, \quad i=1,\ldots,m.
\end{equation}
By Lemma~\ref{lem1} for $i=1,\ldots,m-1$ the function $E_i(\bar\xi)$ is strictly convex as a function of two variables $\xi_i,\xi_{i+1}$ and it follows from (\ref{deg1}) that $\zeta_i=\zeta_{i+1}=0$, $i=1,\ldots,m-1$. Observe that in the case $m=1$ there are no such $i$. In this case we apply (\ref{deg1}) for $i=m$. Taking into account Remark~\ref{rem1}, we find that $E_m(\bar\xi)$ is a strictly convex function of the single variable $\xi_m$, and it follows from (\ref{deg1}) that $\zeta_m=0$. In any case we obtain that the vector $\zeta=0$. Thus, relation (\ref{deg}) can hold only for zero $\zeta$, that is, the matrix $D^2\tilde E(\bar\xi)$ is (strictly) positive definite, and the function $\tilde E(\bar\xi)$ is strictly convex. This completes the proof.
\end{proof}
Propositions~\ref{th1},\ref{th2} imply the main result of this section.

\begin{theorem}\label{th3}
There exists a unique self-similar solution (\ref{3}) of problem (\ref{1}), (\ref{St}), (\ref{2}), and it corresponds to the minimum of strictly convex and coercive function (\ref{5}).
\end{theorem}

\begin{remark}\label{rem2} In paper \cite{Pan2} Stefan problem (\ref{1}), (\ref{St}) was studied in the half-plane $t>0$, $x\in\R$ with Riemann initial condition
\begin{equation}
\label{R}
u(0,x)=\left\{\begin{array}{lr} u_+, & x>0, \\ u_-, & x<0. \end{array}\right.
\end{equation}
Solutions of this problem have the same structure as in (\ref{3}) and correspond to a unique minimum point of the function similar to (\ref{5}) with only the difference that the parameters $\xi_i$ are not necessarily positive and can take arbitrary real values. In this section we mainly follow the scheme of paper \cite{Pan2}.

Remark also that in paper \cite{Pan1} the Stefan-Riemann problem (\ref{1}), (\ref{St}), (\ref{R}) was studied in the case of arbitrary (possibly negative) latent specific heats $d_i$. We found a necessary and sufficient condition for coercivity of $E(\bar\xi)$, as well as a stronger sufficient condition of its strict convexity. The similar results can be obtained for the Stefan-Dirichlet problem (\ref{1}), (\ref{St}), (\ref{2}).
\end{remark}

\section{The case of infinitely many phase transitions}

In this small section we consider the exotic case when the number $m$ of phase transitions is infinite. More precisely, we suppose that the phase transitions temperatures $u_i\ge u_0$ form a strictly increasing sequence, $u_{i+1}>u_i$, $i\in\N$. In (\ref{1}) there are infinitely many phases, the $i$-th phase corresponds to the temperature $u\in (u_i,u_{i+1})$, $i=\{0\}\cup\N$. The parameters $a_i,k_i>0$ are, respectively, the diffusivity constant and the thermal conductivity of the $i$-th phase, $i=0,1,\ldots$; $d_i\ge 0$ is the latent specific heat of the $i$-th phase transition,
$i=1,2,\ldots$. As in the previous section, we assume that $d_1>0$ whenever $u_1=u_0$. We will study the problem (\ref{1}), (\ref{St}), (\ref{2}) with possibly infinite Dirichlet data $u_D=\lim\limits_{i\to\infty} u_i\le+\infty$. A self-similar solution $u=u(\xi)$, $\xi=x/\sqrt{t}$, of this problem is given by expression (\ref{3}), where now $i$ runs over all nonnegative integers. The Stefan conditions (\ref{St}) reduce again to (now infinite) system (\ref{4}), which can be written in the form $\frac{\partial}{\partial\xi_i} E(\bar\xi)=0$, where the functional
\begin{align}\label{5i}
E(\bar\xi)=-\sum_{i=0}^\infty  k_i(u_{i+1}-u_i)\ln (F(\xi_i/a_i)-F(\xi_{i+1}/a_i))+\sum_{i=1}^\infty d_i\xi_i^2/4, \\ \nonumber \bar\xi=(\xi_i)_{i\in\N}\in\Omega,
\end{align}
\[\Omega=\{ \ \bar\xi=(\xi_i)_{i\in\N}\in l_\infty \ | \ \xi_i>\xi_{i+1}>0 \ \forall i\in\N \ \}. \]
Here, as usual, $l_\infty$ is the space of bounded sequences equipped with the norm $\|\bar\xi\|_\infty=\sup |\xi_i|$.
As is well-known, this space is dual to the space of summable sequences $l_1$. Observe also that there is only finite number of terms in (\ref{5i}) depending on a fixed variable $\xi_i$. Therefore, the partial derivatives $\frac{\partial}{\partial\xi_i} E(\bar\xi)$ are well defined whenever the value $E(\bar\xi)$ is finite.

We do not include the natural requirement $\lim\limits_{i\to\infty}\xi_i=0$ in the definition of $\Omega$ because this requirement spoils coercivity of the functional $E$ in the weak-$*$ topology. Observe that the functional $E$ may take the value $+\infty$, moreover, it can happen that $E\equiv+\infty$. Assuming that the latter does not however happen, i.e. the functional $E$ is proper, $E\not\equiv+\infty$, we will show that this functional admits a unique global minimum point. For that we need some nice properties of $E$ collected below.

\begin{proposition}\label{pro1}

(i) The functional $E(\bar\xi)$ is low semi-continuous in weak-$*$ topology;

(ii) It is coercive, that is, the sets $\Omega(c)=\{ \bar\xi\in\Omega \ | \ E(\bar\xi)\le c \ \}$ are weakly-$*$ compact;

(iii) The functional $E(\bar\xi)$ is strictly convex on $\Omega$.
\end{proposition}

\begin{proof}
It is known that weak-$*$ convergence $\bar\xi_n\rightharpoonup \bar\xi$ of a sequence $\bar\xi_n=(\xi_i^n)_{i\in\N}\in l_\infty$, $n\in\N$, is equivalent to uniform boundedness $|\xi_{ni}|\le\const$ and elementwise convergence $\xi_i^n\mathop{\to}\limits_{n\to\infty} \xi_i$ of this sequence. Assume that $\bar\xi_n\in\Omega$, $\bar\xi\in\Omega$, and $\bar\xi_n\rightharpoonup \bar\xi$ weakly-$*$ in $l_\infty$. Then $\xi_i^n\to \xi_i$ as $n\to\infty$ for each $i\in\N$. Since all terms in formula (\ref{5i}) are nonnegative, we can apply Fatou's lemma for sums and conclude that
\begin{align*}
E(\bar\xi)=\sum_{i=0}^\infty  -k_i(u_{i+1}-u_i)\ln (F(\xi_i/a_i)-F(\xi_{i+1}/a_i))+\sum_{i=1}^\infty d_i\xi_i^2/4\le \\ \liminf_{n\to\infty} \{\sum_{i=0}^\infty-k_i(u_{i+1}-u_i)\ln (F(\xi_i^n/a_i)-F(\xi_{i+1}^n/a_i))+\sum_{i=1}^\infty d_i(\xi_i^n)^2/4\}=\liminf_{n\to\infty} E(\bar\xi_n).
\end{align*}
Hence, the functional $E$ is weakly-$*$ low semi-continuous, and (i) is proven.

To prove (ii) we first notice that $\Omega(c)=\emptyset$ if $c\le 0$ and we can suppose that $c>0$. If $\bar\xi=(\xi_i)_{i\in\N}\in\Omega(c)$ then relations (\ref{co1a}), (\ref{co1b}) hold for all $i\in\{0\}\cup\N$. As in the proof of Proposition~\ref{th1}, we derive from these relations that
\[
\xi_1\le r_2=\left\{\begin{array}{lcr}a_0F^{-1}(1-e^{-c/(k_0(u_1-u_0))}) & , & u_0<u_1, \\
(4c/d_1)^{1/2} & , & u_0=u_1. \end{array}\right.
\]
Further, it follows from (\ref{co1a}) that for every $i\in\N$
\[
F(\xi_i/a_i)-F(\xi_{i+1}/a_i)\ge \delta_i\doteq\exp(-c/\alpha_i)>0,
\]
where $\alpha_i=k_i(u_{i+1}-u_i)>0$. Since $F(\xi)$ is Lipschitz with constant $1$, this implies the inequalities
\[
\xi_i-\xi_{i+1}\ge a_i\delta_i>0, \quad i\in\N.
\]
Hence, $\Omega(c)$ is contained in the set
\[
K=\{ \bar\xi\in l_\infty \ | \ 0\le\xi_i\le r_2, \ \xi_i-\xi_{i+1}\ge a_i\delta_i \ \forall i\in\N \ \}.
\]
Obviously, this set is bounded and weakly-$*$ closed in $l_\infty$. By the Banach-Alaoglu theorem
the set $K$ is weakly-$*$ compact. Since $E$ is weakly-$*$ semi-continuous, the set $\Omega(c)$ is a closed subset of $K$ and, therefore, is compact. Coercivity of $E$ is proved.

To prove (iii), we observe that the functional $E(\bar\xi)$ is convex as a sum of the convex functionals $-k_i(u_{i+1}-u_i)\ln (F(\xi_i/a_i)-F(\xi_{i+1}/a_i))$ and $d_i\xi_i^2/4$. By the same reason for each $m\in\N$ the functional
\[
R_m(\bar\xi)=-\sum_{i=m}^\infty  k_i(u_{i+1}-u_i)\ln (F(\xi_i/a_i)-F(\xi_{i+1}/a_i))+\sum_{i=m+1}^\infty d_i\xi_i^2/4
\]
is convex while the function
\[ E_m(\bar\xi_m)=-\sum_{i=0}^{m-1} k_i(u_{i+1}-u_i)\ln (F(\xi_i/a_i)-F(\xi_{i+1}/a_i))+\sum_{i=1}^m d_i\xi_i^2/4, \quad
\bar\xi_m=(\xi_1,\ldots,\xi_m)
\]
is strictly convex on $\R^m$, this can be established on the base of Lemma~\ref{lem1} in the same way as in the proof of Proposition~\ref{th2}. It is clear that $E(\bar\xi)=E_m(\bar\xi)+R_m(\bar\xi_m)$. Now we take different points $\bar\xi^1,\bar\xi^2\in\Omega$ and $\alpha\in (0,1)$. Since $\bar\xi^1\not=\bar\xi^2$ then we can find such $m\in\N$ that
$\bar\xi^1_m\not=\bar\xi^2_m$, where $\bar\xi^n_m$, $n=1,2$, are the vectors in $\R^m$ formed by the first $m$ elements of $\bar\xi^n$. Since the function $E_m$ is strictly convex then
\[ E_m((1-\alpha)\bar\xi^1_m+\alpha\bar\xi^2_m)<(1-\alpha)E_m(\bar\xi^1_m)+\alpha E_m(\bar\xi^2_m)\]
while
\[ R_m((1-\alpha)\bar\xi^1+\alpha\bar\xi^2)\le (1-\alpha)R_m(\bar\xi^1)+\alpha R_m(\bar\xi^2)\]
by the convexity of $R_m$. Therefore,
\[ E((1-\alpha)\bar\xi^1+\alpha\bar\xi^2)<(1-\alpha)E(\bar\xi^1)+\alpha E(\bar\xi^2).\]
This proves the strict convexity of $E(\bar\xi)$.
\end{proof}

In the case $E(\bar\xi)\not\equiv +\infty$ Proposition~\ref{pro1} allows to establish existence of a solution to the problem (\ref{1}), (\ref{St}), (\ref{2}) with infinite number of phase transitions. More precisely, the following statement holds.

\begin{theorem}\label{th1i}
Assume that the potential $E(\bar\xi)$ is a proper functional, i.e. $N=\inf E(\bar\xi)<+\infty$. Then there exists a unique
minimum point $\bar\xi^0=(\xi^0_i)_{i\in\N}\in\Omega$ of $E(\bar\xi)$, i.e., $E(\bar\xi^0)=N$. Moreover, $\lim\limits_{i\to\infty}\xi^0_i=0$, and function (\ref{3})
\begin{align}\label{3i}
u(\xi)=u_i+\frac{u_{i+1}-u_i}{F(\xi^0_{i+1}/a_i)-F(\xi^0_i/a_i)}(F(\xi/a_i)-F(\xi^0_i/a_i)), \\ \nonumber
\xi^0_{i+1}<\xi<\xi^0_i, \ i=0,1,\ldots
\end{align}
is a solution of (\ref{1}), (\ref{St}), (\ref{2}).
\end{theorem}

\begin{proof}
We define for $n\in\N$ the sub-level sets $K_n=\Omega(N+1/n)$ where $E\le N+1/n$. Then these sets are nonempty. By Proposition~\ref{pro1} they are weakly-$*$ compact. Obviously, $K_{n+1}\subset K_n$ $\forall n\in\N$. By Cantor's intersection theorem there exists a point $\bar\xi^0\in \bigcap\limits_{n\in\N} K_n$. Then, $N\le E(\bar\xi^0)\le N+1/n$ for all $n\in\N$, which implies that $E(\bar\xi^0)=N$, and $\bar\xi^0$ is a point of global minimum of $E$. Uniqueness of this point directly follows from the strict convexity of $E$ stated in Proposition~\ref{pro1}(iii). It only remains to prove that the sequence
$\bar\xi^0=(\xi^0_i)_{i\in\N}$ vanishes. Assuming the contrary, we will have $\displaystyle \lim_{i\to\infty} \xi^0_i=\inf_{i\in\N}\xi^0_i=r>0$. Then the sequence $\bar\xi^r=\bar\xi^0-r$ with components $\xi^0_i-r$ lies in $\Omega$.
Since $1-F((\xi^0_1-r)/a_0)>1-F(\xi^0_1)/a_0$ and for all $i\in\N$ $(\xi^0_i-r)^2<(\xi^0_i)^2$,
\begin{align*}
F((\xi^0_i-r)/a_i)-F((\xi^0_{i+1}-r)/a_i)=\frac{1}{\sqrt{\pi}}\int_{(\xi^0_{i+1}-r)/a_i}^{(\xi^0_i-r)/a_i} e^{-s^2/4}ds>\\ \frac{1}{\sqrt{\pi}}\int_{\xi^0_{i+1}/a_i}^{\xi^0_i/a_i} e^{-s^2/4}ds= F(\xi^0_i/a_i)-F(\xi^0_{i+1}/a_i),
\end{align*}
all the terms in expression (\ref{5i}) became smaller if we replace $\bar\xi^0$ by $\bar\xi^r$. Therefore,
$E(\bar\xi^r)<E(\bar\xi^0)=N$, which is impossible. Hence, $\lim\limits_{i\to\infty}\xi^0_i=0$. Since $\bar\xi^0$ is a minimum point of $E$ then $\frac{\partial}{\partial\xi_i}E(\bar\xi^0)=0$ for all $i\in\N$ and Stefan conditions (\ref{St}) are satisfied. We conclude that (\ref{3i}) is a solution of our problem (\ref{1}), (\ref{St}), (\ref{2}).
\end{proof}

\begin{example}
Assume that $k_i=a_i^2=1$, $i\ge 0$; $d_i=0$, $i>0$. Then Stefan conditions (\ref{St}) simply means that $u(\xi)$ is a self-similar solution of the heat equation $u_t=u_{xx}$. Therefore, $u(\xi)=C_1F(\xi)+C_2$, $C_1,C_2=\const$, and $u(\xi)$ is a bounded function. In particular, a solution of (\ref{1}), (\ref{St}), (\ref{2}) with $u_D=\lim\limits_{i\to\infty} u_i=+\infty$ does not exists.
As is easy to verify, in this case $E(\bar\xi)\equiv +\infty$, so that the assumption of Theorem~\ref{th1i} is violated. On the other hand, if $u_D<+\infty$ then a unique solution of (\ref{1}), (\ref{St}), (\ref{2}) has the form
$u(\xi)=u_D-(u_D-u_0)F(\xi)$. Solving the equations $u(\xi)=u_i$, we find $\xi_i=F^{-1}((u_D-u_i)/(u_D-u_0))$. Hence,
\begin{align*}
E(\bar\xi)=-\sum_{i=0}^\infty (u_{i+1}-u_i)\ln(F(\xi_i)-F(\xi_{i+1}))=\sum_{i=0}^\infty (u_{i+1}-u_i)(\ln(u_D-u_0)-\ln(u_{i+1}-u_i)) \\ =(u_D-u_0)\ln(u_D-u_0)-\sum_{i=0}^\infty (u_{i+1}-u_i)\ln(u_{i+1}-u_i).
\end{align*}
As follows from Theorem~\ref{th1i} and the uniqueness of our solution, this value is the minimum value of $E$ whenever the functional $E(\bar\xi)$ is proper. In particular, $E(\bar\xi)\equiv +\infty$ if (and only if) the series $\displaystyle \sum_{i=0}^\infty (u_{i+1}-u_i)\ln(u_{i+1}-u_i)$ is divergent.
\end{example}

\section{Stefan problem with Neumann boundary condition}

Now we return to the case of finite $m$ and consider the Stefan problem (\ref{1}), (\ref{St}),
with the constant initial data $u_0$ and with Neumann boundary condition:
\begin{equation}\label{Neu}
u(0,x)=u_0 \ \forall x>0, \quad \alpha(u)_x(t,0)=t^{-1/2} b_N \ \forall t>0,
\end{equation}
where $\alpha(u)$ is the diffusion function in equation (\ref{diff}), now defined on infinite interval $[u_0,+\infty)$ (we take $u_{m+1}=+\infty$), and $b_N<0$ is a constant. The specific form of Neumann boundary data is connected with the requirement of invariance of our problem under the scaling transformations $(t,x)\to (\lambda^2 t,\lambda x)$, $\lambda>0$. This allows to concentrate on the study of self-similar solutions $u=u(x/\sqrt{t})$ of the problem (\ref{1}), (\ref{St}), (\ref{Neu}).
For such solutions conditions (\ref{Neu}) reduce to the requirements
\begin{equation}\label{Neu1}
\alpha(u)'(0)=b_N, \quad u(+\infty)=u_0.
\end{equation}
Since $\alpha(u)$ is a strictly increasing function and $b_N<0$, we will assume that the function $u(\xi)$ decreases. The case $b_N>0$ corresponds to an increasing $u(\xi)$ and can be treated similarly. For homogeneous Neumann problem $b_N=0$ there is only the constant solution $u\equiv u_0$.
Let $u_i$, $i=1,\ldots,m$, be all phase transition temperatures in the interval $[u_0,+\infty)$ so that
$u_0\le u_1<\cdots<u_m<u_{m+1}=+\infty$. As in Section~\ref{sec1}, we suppose that the latent specific heat $d_1>0$ if $u_1=u_0$.
Assume that $u=u(\xi)$ is a decreasing self-similar solution of (\ref{1}), (\ref{St}), (\ref{Neu}). Then $u(0)>u_0$ and there is an integer $n$, $0\le n\le m$, such that $u_n<u(0)\le u_{n+1}$. We call this number $n$ (i.e., the number of phase transitions) a type of solution $u$. The Neumann condition for a solution of type $n$ reads $k_nu'(0)=b_N$
(notice that $k_nu_x$ is exactly the heat flow through the boundary point $x=0$).
A solution of type $0$ does not contain free boundaries and can be found by the formula
\begin{equation}\label{t0}
u(\xi)=u_0+\frac{a_0}{k_0}b_N\sqrt{\pi}(F(\xi/a_0)-1).
\end{equation}
As is easy to verify, $k_0u'(0)=b_N$, $u(+\infty)=u_0$ and requirement (\ref{Neu1}) is satisfied. By easy computation we find $u(0)=u_0-a_0b_N\sqrt{\pi}/k_0$, therefore, the necessary and sufficient condition for existence of a solution (\ref{t0}) is the
inequality $u_0-a_0b_N\sqrt{\pi}/k_0\le u_1$, which can be written in the form
\begin{equation}\label{con0}
-b_N\le \gamma_1\doteq\frac{k_0(u_1-u_0)}{a_0\sqrt{\pi}}.
\end{equation}
In particular, a solution of type $0$ does not exist if $u_1=u_0$.
A solution of type $n>0$ has structure similar to (\ref{3}) (with $m=n$)
\begin{align}\label{tki}
u(\xi)=u_i+\frac{u_{i+1}-u_i}{F(\xi_{i+1}/a_i)-F(\xi_i/a_i)}(F(\xi/a_i)-F(\xi_i/a_i)), \quad
\xi_{i+1}<\xi<\xi_i, \ i=0,\ldots,n-1, \\
\label{tkn}
u(\xi)=u_n+\frac{a_n}{k_n}b_N\sqrt{\pi}(F(\xi/a_n)-F(\xi_n/a_n)), \quad 0<\xi<\xi_n.
\end{align}
The necessary (but not sufficient, as we will soon realize) condition $u(0)\le u_{n+1}$ has the form
\begin{equation}\label{conn}
-b_N\sqrt{\pi}F(\xi_n/a_n)\le k_n(u_{n+1}-u_n)/a_n
\end{equation}
(if $n=m$ then it is always fulfilled since $u_{m+1}=+\infty$).

Firstly, we investigate the uniqueness of a solution of our Stefan-Neumann problem. If $u=u(t,x)$ is a solution of
(\ref{1}), (\ref{St}), (\ref{Neu}) then the even extension $\tilde u(t,x)=u(t,|x|)$ is continuous on the half-plane $\Pi=(0,+\infty)\times\R$ and it is a distributional solution of (\ref{diff}) in a domain $(t,x)\in\Pi$, $x\not=0$.
Now, the increasing functions $\beta(u)$, $\alpha(u)$ are defined on $[u_0,+\infty)$ by the same relations as in (\ref{diff}):
$\alpha'(u)=k_i$, $\beta'(u)=k_i/a_i^2$ on intervals $(u_i,u_{i+1})$, $i=0,\ldots,m$; $\alpha(u_i+)-\alpha(u_i-)=0$,
$\beta(u_i+)-\beta(u_i-)=d_i$, $i=1,\ldots,m$.

Since the jump of $\alpha(\tilde u)_x$ at the line $x=0$
\[[\alpha(\tilde u)_x]=\alpha(\tilde u)_x(t,0+)-\alpha(\tilde u)_x(t,0-)=2\alpha(u)_x(t,0+)=2b_Nt^{-1/2},\]
it follows that
\begin{equation}\label{diff1}
\beta(\tilde u)_t-\alpha(\tilde u)_{xx}=-2b_Nt^{-1/2}\delta(x) \ \mbox{ in } \D'(\Pi),
\end{equation}
where $\delta(x)$ is the Dirac $\delta$-function. Besides, the initial condition $\beta(\tilde u)(0,x)=\beta(u_0)$ holds in the sense of point-wise convergence as $t\to 0+$ for all $x\not=0$.
Notice that the function $\beta(u)$ has a jump at $u=u_0$ if $u_1=u_0$. In this case $\beta(\tilde u)(t,x)=\beta(u_0)$
in the set $|x|>\xi_1\sqrt{t}$ by (\ref{tki}) with $i=0$.

Now, we are ready to prove the uniqueness.

\begin{theorem}\label{thun}
A solution $u=u(\xi)$ of the problem (\ref{1}), (\ref{St}), (\ref{Neu}) having the self-similar form (\ref{tki}), (\ref{tkn}) (for some $n\in \overline{0,m}$, which may depend on a solution) is unique.
\end{theorem}

\begin{proof}
Assume that $u_1$, $u_2$ are two solutions of (\ref{1}), (\ref{St}), (\ref{Neu}) (possibly, of different types), and
$\tilde u_i(t,x)=u_i(t,|x|)$, $i=1,2$. In view of (\ref{diff1}) we have
\begin{equation}\label{7}
(\beta(\tilde u_1)-\beta(\tilde u_2))_t-(\alpha(\tilde u_1)-\alpha(\tilde u_2))_{xx}=0 \ \mbox{ in } \D'(\Pi).
\end{equation}
Remark that for all $u,v\ge u_0$, $u\not=v$
\begin{equation}\label{8}
0<\frac{\alpha(u)-\alpha(v)}{\beta(u)-\beta(v)}\le \max a_i^2.
\end{equation}
We introduce the function $P=P(t,x)=\beta(\tilde u_1)-\beta(\tilde u_2)\in L^\infty(\Pi)$.

As follows from the representation (\ref{tki}) with $i=0$, for large $\xi=|x|/\sqrt{t}$
\begin{equation}\label{9}
|P(t,x)|=|\beta(\tilde u_1)-\beta(\tilde u_2)|=\frac{k_0}{a_0^2}|u_1(\xi)-u_2(\xi)|=c(1-F(\xi/a_0)), \quad c=\const
\end{equation}
(moreover, $c=0$ in the case $u_1=u_0$).
By the L'H\^{o}pital's rule and the identity ${F''(y)=-\frac{y}{2}F'(y)}$
\[
\lim_{y\to+\infty}\frac{2y^{-1}F'(y)}{1-F(y)}=2\lim_{y\to+\infty}\frac{y^{-2}F'(y)-y^{-1}F''(y)}{F'(y)}=
2\lim_{y\to+\infty} [y^{-2}+y^{-1}y/2)]=1.
\]
Therefore, $1-F(y)\sim 2y^{-1}F'(y)=2y^{-1}e^{-y^2/4}$ as $y\to+\infty$. This implies that for large~$\xi$ \[|P(t,x)|\le\const\cdot\xi^{-1}e^{-\xi^2/(4a_0^2)}.\]
Therefore, $P(t,\cdot)\in L^2(\R)$ and $\|P(t,\cdot)\|_2=ct^{1/4}$, $c=\const$. In particular, $P(t,\cdot)\to 0$
as $t\to 0+$ in $L^2(\R)$.
Similar statements hold for the function $Q=\alpha(\tilde u_1)-\alpha(\tilde u_2)$ because $Q=CP$, where $C=C(t,x)=(\alpha(\tilde u_1)-\alpha(\tilde u_2))/(\beta(\tilde u_1)-\beta(\tilde u_2))$ (if $\tilde u_1=\tilde u_2$ we set $C=0$) is a nonnegative bounded function, in view of (\ref{8}).

Hence, we can apply (\ref{7}) to a test function $f=f(t,x)$ from the Sobolev space $W_2^{1,2}(\Pi_T)$, $\Pi_T=(0,T)\times\R$ (so that $f,f_t,f_x,f_{xx}\in L^2(\Pi_T)$) such that $f(T,x)=0$. As a result, we obtain the relation
\begin{equation}\label{m5}
\int_{\Pi_T} P(t,x)[f_t+Cf_x]dtdx=0.
\end{equation}
Let $F(t,x)\in C_0^1(\Pi_T)$, $\varepsilon>0$, and $f^\varepsilon=f^\varepsilon(t,x)\in W_2^{1,2}(\Pi_T)$ be a solution of
the backward Cauchy problem
\[
f_t+(C+\varepsilon)f_{xx}=F, \quad f(T,x)=0.
\]
As is demonstrated in \cite{Kam,LSU} in general multidimensional case, such a solution exists and satisfies the estimate
\begin{equation}\label{m7}
\|f_{xx}^\varepsilon\|_2\le C_0/\sqrt{\varepsilon},
\end{equation}
where $C_0$ is a constant independent of $\varepsilon$ and $\|\cdot\|_2=\|\cdot\|_{L^2(\Pi_T)}$. It follows from (\ref{m5}) with $f=f^\varepsilon$ that
\begin{equation}\label{m8}
\int_{\Pi_T} P(t,x)F(t,x)dtdx=\varepsilon\int_{\Pi_T} P(t,x)f_{xx}^\varepsilon(t,x)dtdx.
\end{equation}
By the Cauchy-Bunyakovsky inequality and (\ref{m7})
\[\left|\int_{\Pi_T} P(t,x)f_{xx}^\varepsilon(t,x)dtdx\right|\le\|P\|_2\|f_{xx}^\varepsilon\|_2\le C_0\|P\|_2/\sqrt{\varepsilon},\]
and the right-hand side of (\ref{m8}) vanishes as $\varepsilon\to 0$. We conclude that
\[\int_{\Pi_T} P(t,x)F(t,x)dtdx=0\]
for all $F(t,x)\in C_0^1(\Pi_T)$ and all $T>0$. This means that $P=0$ a.e. on $\Pi$, that is, $u_1=u_2$.
\end{proof}

Assume that $u(\xi)$ is a solution (\ref{tki}), (\ref{tkn}) of type $n$.
On a phase transition lines $\xi=\xi_i$, the Stefan condition reads
\begin{align}\label{sysi}
d_i\xi_i/2+k_i\frac{(u_{i+1}-u_i)F'(\xi_i/a_i)}{a_i(F(\xi_{i+1}/a_i)-F(\xi_i/a_i))}- \nonumber\\
k_{i-1}\frac{(u_i-u_{i-1})F'(\xi_i/a_{i-1})}{a_{i-1}(F(\xi_i/a_{i-1})-F(\xi_{i-1}/a_{i-1}))}=0, \quad i=1,\ldots,n-1, \\
\label{sysn}
d_n\xi_n/2+b_N\sqrt{\pi} F'(\xi_n/a_n)- k_{n-1}\frac{(u_n-u_{n-1})F'(\xi_n/a_{n-1})}{a_{n-1}(F(\xi_n/a_{n-1})-F(\xi_{n-1}/a_{n-1}))}=0, \quad i=n.
\end{align}
Like in the case of Dirichlet boundary condition, this system turns out to be gradient one, it coincides with the equality $\nabla E_n=0$, where the function
\begin{align}\label{E}
E_n(\bar\xi)=-\sum_{i=0}^{n-1}k_i(u_{i+1}-u_i)\ln (F(\xi_i/a_i)-F(\xi_{i+1}/a_i))\nonumber\\
+a_nb_N\sqrt{\pi}F(\xi_n/a_n)+\frac{1}{4}\sum_{i=1}^n d_i\xi_i^2, \quad \bar\xi=(\xi_1,\ldots,\xi_n)\in\Omega_n,
\end{align}
$\Omega_n$ is an open convex cone in $\R^n$ consisting of vectors with strictly decreasing positive coordinates.
Remark that $F''(s)=-s/2F'(s)<0$ for all $s>0$. Since $b_N<0$, this implies that the term $a_nb_N\sqrt{\pi}F(\xi_n/a_n)$ is strictly convex function of single variable $\xi_n$ on the interval $[0,+\infty)$. In the same way as in the proof of Proposition~\ref{th2}
we find that the function
\[
\tilde E_n(\bar\xi)=-\sum_{i=0}^{n-1}k_i(u_{i+1}-u_i)\ln (F(\xi_i/a_i)-F(\xi_{i+1}/a_i))+\frac{1}{4}\sum_{i=1}^n d_i\xi_i^2
\]
is strictly convex on a cone
\[\bar\Omega_n=\{ \ \bar\xi=(\xi_1,\ldots,\xi_n)\in\R^n \ | \ \xi_1>\cdots>\xi_n\ge 0 \ \}\supset\Omega_n,\]
consisting of points with strictly decreasing nonnegative coordinates. Since
\[E_n(\bar\xi)=\tilde E_n(\bar\xi)+a_nb_N\sqrt{\pi}F(\xi_n/a_n),\] the function $E_n(\bar\xi)$ is strictly convex on $\bar\Omega_n$ as well.
Let us demonstrate that this function is coercive on $\bar\Omega_n$.

\begin{proposition}\label{th4}
For all $c\in\R$ the set $\bar\Omega_n(c)=\{ \ \bar\xi\in\bar\Omega_n \ | \ E_n(\bar\xi)\le c \ \}$ is compact.
\end{proposition}

\begin{proof}
Suppose that $\bar\xi\in\bar\Omega_n$, $E_n(\bar\xi)\le c$. Then
\begin{align*}
\tilde E_n(\bar\xi)\doteq -\sum_{i=0}^{n-1}k_i(u_{i+1}-u_i)\ln (F(\xi_i/a_i)-F(\xi_{i+1}/a_i))+ \\
\frac{1}{4}\sum_{i=1}^n d_i\xi_i^2=E(\bar\xi)-a_nb_N\sqrt{\pi}F(\xi_n/a_n)\le c_1\doteq c-a_nb_N\sqrt{\pi}.
\end{align*}
Since all the terms of the left-hand side of this inequality are nonnegative, we obtain the relations
\begin{align}\label{co2a}
-k_i(u_{i+1}-u_i)\ln (F(\xi_i/a_i)-F(\xi_{i+1}/a_i))\le c_1, \ i=0,\ldots,n-1, \\
\label{co2b}
d_i\xi_i^2/4\le c_1, \ i=1,\ldots,n,
\end{align}
the same as inequalities (\ref{co1a}), (\ref{co1b}). As follows from (\ref{co2a}), (\ref{co2b}), the set $\bar\Omega_n(c)=\emptyset$ if $c_1<0$. Therefore, we may (and will) suppose that $c_1\ge 0$. Arguing as in the proof of Proposition~\ref{th1}, we derive from (\ref{co2a}), (\ref{co2b}) the bounds
\begin{align*}
\xi_1\le r_2=\left\{\begin{array}{lcr}a_0F^{-1}(1-e^{-c_1/(k_0(u_1-u_0))}) & , & u_0<u_1, \\
(4c_1/d_1)^{1/2} & , & u_0=u_1, \end{array}\right. \\
(\xi_i-\xi_{i+1})/a_i\ge \delta=\exp(-c_1/\alpha)>0, \quad i=1,\ldots,n-1,
\end{align*}
where $\displaystyle\alpha=\min_{i=1,\ldots,m-1}k_i(u_{i+1}-u_i)>0$. Thus, the set $\bar\Omega_n(c)$ is contained in a compact
\[
K=\{ \ \bar\xi=(\xi_1,\ldots,\xi_n)\in\R^n \ | \ r_2\ge\xi_1\ge\cdots\ge\xi_n\ge 0, \ \xi_i-\xi_{i+1}\ge\delta a_i \ \forall i=1,\ldots,n-1 \ \}.
\]
Since $E_n(\bar\xi)$ is continuous on $K$, the set $\bar\Omega_n(c)$ is a closed subset of $K$ and therefore is compact. This completes the proof.
\end{proof}

It follows from Proposition~\ref{th4} and the strict convexity of function $E_n$ that there exists a point
$\bar\xi^n=(\xi_1^n,\ldots,\xi_n^n)\in\bar\Omega_n$ of global minimum of $E_n$, and it is a unique local minimum of this function. There are two possible cases:

A) $\bar\xi^n\in\Omega_n$, i.e. $\xi_n^n>0$. If, in addition, condition (\ref{conn}) is satisfied then there exists a unique solution (\ref{tki}), (\ref{tkn}) of type $n$
with $\xi_i=\xi_i^n$, $i=1,\ldots,n$;

B) $\bar\xi^n\notin\Omega_n$, i.e. $\xi_n^n=0$. Then a solution of type $n$ does not exist. Let us investigate this case more precisely. The necessary and sufficient conditions for the point $\bar\xi^n=(\xi_1^n,\ldots,\xi_{n-1}^n,0)$ to be a minimum point of $E_n(\bar\xi)$ are the following
\begin{align}
\label{mc1}
\frac{\partial}{\partial\xi_i} E_n(\bar\xi^n)=0, \ i=1,\ldots,n-1, \\
\label{mc2}
\frac{\partial}{\partial\xi_n} E_n(\bar\xi^n)\ge 0,
\end{align}
where condition (\ref{mc1}) appears only if $n>1$. Notice that for such $n$
\begin{align*}
E_n(\xi_1,\ldots,\xi_{n-1},0)=-\sum_{i=0}^{n-1}k_i(u_{i+1}-u_i)\ln (F(\xi_i/a_i)-F(\xi_{i+1}/a_i))
+\frac{1}{4}\sum_{i=1}^{n-1} d_i\xi_i^2, \\ \xi_n=0, \ (\xi_1,\ldots,\xi_{n-1})\in\Omega_{n-1}.
\end{align*}
We see that $E(\xi_1,\ldots,\xi_{n-1})\doteq E_n(\xi_1,\ldots,\xi_{n-1},0)$ coincides with function (\ref{5}) with $m=n-1$, corresponding to
Stefan-Dirichlet problem (\ref{1}), (\ref{St}), (\ref{2}) with $u_D=u_n$.
Relation (\ref{mc1}) means that $\nabla E(\xi_1,\ldots,\xi_{n-1})=0$, that is, $(\xi_1^0,\ldots,\xi_{n-1}^0)\in\Omega_{n-1}$ is a unique minimal point of $E(\xi_1,\ldots,\xi_{n-1})$. According to Theorem~\ref{th3}, the coordinates $\xi_i^n$, $i=1,\ldots,n-1$, coincide with the phase transition parameters $\xi_i$ of the unique solution (\ref{3}) of problem (\ref{1}), (\ref{St}), (\ref{2}) with $u_D=u_n$ (in particular, they do not depend on the Neumann data $b_N$ and on parameters $a_i$, $k_i$, $d_i$ with $i\ge n$). As is easy to calculate, condition (\ref{mc2}) reads
\begin{equation}\label{cB}
\frac{k_{n-1}(u_n-u_{n-1})}{\sqrt{\pi}a_{n-1}F(\xi_{n-1}^n/a_{n-1})}+b_N\ge 0.
\end{equation}
This formula remains valid also for $n=1$, in this case one have to take $\xi_{n-1}^n=\xi_0^1=+\infty$, so that
$F(\xi_{n-1}^n/a_{n-1})=F(+\infty)=1$. Under requirement (\ref{cB}) the case B) is realised so that a solution of type $n$ does not exist. More precisely, the following statements hold for $n\ge 1$.

\begin{lemma}\label{lem2}
(i) If $\displaystyle 0<-b_N\le\gamma_n=\frac{k_{n-1}(u_n-u_{n-1})}{\sqrt{\pi}a_{n-1}F(\xi_{n-1}^n/a_{n-1})}$
then a solution of type $n$ does not exist;

(ii) If $-b_N>\gamma_n$ and $-b_N-\gamma_n$ is small enough then a solution of type $n$ exists. Moreover, if
$\gamma_n<\gamma_{n+1}$ then a solution of type $n$ exists for each Neumann data $b_N$ such that $\gamma_n<-b_N\le\gamma_{n+1}$.
\end{lemma}

\begin{proof}
If $-b_N\le\gamma_n$ then relation (\ref{cB}) holds and the first statement follows.
If $n=m$ then condition (\ref{conn}) is always satisfied and a unique solution of type $m$ exists for $\gamma_m<-b_N<\gamma_{m+1}\doteq +\infty$. Hence, it remains to prove (ii) in the case $1\le n<m$.
We notice that by the strict convexity of function (\ref{E}) its minimum point $\bar\xi^n$ depends continuously on the parameter $r=-b_N$. In  particular, the last coordinate $\xi_n^n=\xi_n^n(r)$ is a continuous function of $r$. Since $\xi_n^n(\gamma_n)=0$ then for sufficiently small $-b_N-\gamma_n>0$ the left-hand side of relation (\ref{conn}) can be made so small that this relation is satisfied while condition (\ref{cB}) is violated. We conclude that there exists a unique solution (\ref{tki}), (\ref{tkn}) of type $n$ with $\xi_i=\xi_i^n(-b_N)$, $i=1,\ldots,n$. In the case $\gamma_n<\gamma_{n+1}$ we introduce the function
$\varphi(r)=k_n(u_{n+1}-u_n)/a_n-r\sqrt{\pi}F(\xi_n^n(r)/a_n)$. This function is continuous on $(0,+\infty)$ and $\varphi(r)=k_n(u_{n+1}-u_n)/a_n>0$ for $0<r\le\gamma_n$ (since $\xi_n^n(r)=0$ for such $r$). If $\varphi(r)=0$ with some $r>\gamma_n$ then for $b_N=-r$ relation (\ref{conn}) holds with equality sign, this means that the solution (\ref{tki}), (\ref{tkn}) satisfies the property $u(0)=u_{n+1}$. Therefore,
$u=u(\xi)$ is a unique solution to Stefan-Dirichlet problem (\ref{1}), (\ref{St}), (\ref{2}) with $u_D=u_{n+1}$. By the definition of numbers $\gamma_k$, $k=1,\ldots,n$, we conclude that $\displaystyle r=\frac{k_n(u_{n+1}-u_n)}{\sqrt{\pi}a_nF(\xi_n^{n+1}/a_n)}=\gamma_{n+1}$. Thus, $\gamma_{n+1}$ is the unique zero of the continuous function $\varphi(r)$. Therefore, $\varphi(r)$ keeps the positive sign on the segment $[\gamma_n,\gamma_{n+1})$
and $\varphi(r)=k_n(u_{n+1}-u_n)/a_n-r\sqrt{\pi}F(\xi_n^n(r)/a_n)\ge 0$ for $\gamma_n<r\le\gamma_{n+1}$. Substituting $r=-b_N$, we obtain that condition (\ref{conn}) is satisfied together with the condition $-b_N>\gamma_n$. Thus, a solution of type $n$ exists.
\end{proof}

\begin{remark}\label{rem3}
In addition to the statement (ii) we notice that condition (\ref{con0}), sufficient for the existence of a solution of type $0$, can be written in the form
\[\gamma_0\doteq 0<-b_N\le\gamma_1\]
Certainly, the latter condition has sense only if $u_1>u_0$, otherwise $\gamma_1=0$.
\end{remark}

Now we are going to demonstrate that actually the requirement $\gamma_n<\gamma_{n+1}$ is always satisfied.

\begin{lemma}\label{lem3}
The sequence $\gamma_n$, $n=1,\ldots,m$, strictly increases. Besides, the condition
\begin{equation}\label{tn}
\gamma_n<-b_N\le \gamma_{n+1}
\end{equation}
is necessary and sufficient for existence of a solution of type $n$.
\end{lemma}

\begin{proof}
We will prove that for all $k=0,\ldots,m-1$
\begin{equation}\label{mon}
\gamma_i<\gamma_{i+1} \quad \mbox{ for all } i\in\N, 1\le i\le k.
\end{equation}
We use the induction in $k$. If $k=0$, the set of $i$ in (\ref{mon}) is empty and there is nothing to prove.
Assuming that (\ref{mon}) holds for $k=n-1$, we have to prove it for $k=n$. For that we only need to establish that $\gamma_n<\gamma_{n+1}$. If this inequality is wrong, then either $\gamma_{n+1}=\gamma_n$ or
$\gamma_i\le\gamma_{n+1}<\gamma_{i+1}$ for some $i\in\overline{0,n-1}$ (where we agree that $\gamma_0=0$ and use the induction assumption). In the former case, by Lemma~\ref{lem2}(ii) for some $-b_N>\gamma_{n+1}=\gamma_n$ there exist solutions of problem (\ref{1}), (\ref{St}), (\ref{Neu}) of both types $n$, $n+1$, which contradicts to the uniqueness statement of Theorem~\ref{thun}.
Now we consider the case $\gamma_i\le\gamma_{n+1}<\gamma_{i+1}$. By Lemma~\ref{lem2}(ii) and Remark~\ref{rem3} we find that for some $b_N$ such that
$\gamma_{n+1}<-b_N<\gamma_{i+1}$ there exist solutions of problem (\ref{1}), (\ref{St}), (\ref{Neu}) of different types $i$ and $n+1$. But this is impossible in view of Theorem~\ref{thun}.
We conclude that $\gamma_{n+1}>\gamma_n$ as required.

To prove the second statement, we remark that by Lemma~\ref{lem2} condition (\ref{tn}) is sufficient for existence of a solution of type $n$. Conversely, if there exists a solution of problem (\ref{1}), (\ref{St}), (\ref{Neu}) having type $n$ then $-b_N>\gamma_n$ by Lemma~\ref{lem2}(i). If $-b_N>\gamma_{n+1}$ then necessarily $n<m$ and there exists such $k>n$ that
$\gamma_k<-b_N\le\gamma_{k+1}$. By Lemma~\ref{lem2} there exists another solution of problem (\ref{1}), (\ref{St}), (\ref{Neu}) that has type $k>n$. By the uniqueness this is impossible. We conclude that $-b_N\le\gamma_{n+1}$. Hence the condition (\ref{tn}) is necessary for existence of a solution of type $n$.
\end{proof}

In view of Theorem~\ref{thun} and Lemmas~\ref{lem2}, \ref{lem3}, we establish the following main results on correctness of problem (\ref{1}), (\ref{St}), (\ref{Neu}).

\begin{theorem}\label{th5} For any Neumann data $b_N<0$ the exists a unique solution of the Stefan-Neumann problem
(\ref{1}), (\ref{St}), (\ref{Neu}). The type $n$ of this solution is determined by the condition $\gamma_n<-b_N\le\gamma_{n+1}$
\end{theorem}

In the case $m=1$, $u_1>u_0$ the parameter $\displaystyle \gamma_1=\frac{k_0(u_1-u_0)}{a_0\sqrt{\pi}}$. Hence, in the case
$0<-b_N\le \gamma_1$ a unique solution $u=u(\xi)$ of (\ref{1}), (\ref{St}), (\ref{Neu}) has type $0$ while in the case $-b_N>\gamma_1$ it has type $1$.

\begin{remark}\label{rem4}
The form of equation (\ref{1}) may induce someone to conclude that a more natural Neumann condition is the following one:
\begin{equation}\label{Neu2}
a_n^2u_x(t,0)=bt^{-1/2}
\end{equation}
whenever $u$ is a solution of type $n$. But this problem is incorrect. Generally, neither existence nor uniqueness holds. In fact, consider the simple case $m=1$, $u_1>u_0$ discussed above. Since for a solution of type $n$
\[ b_N=k_nu_x(t,0)t^{1/2}=\frac{k_n}{a_n^2}b,\] we find that criteria for existence of a solution of type $0$, $1$ are, respectively,   $\displaystyle -b\le \frac{a_0^2}{k_0}\gamma_1$, $\displaystyle-b>\frac{a_1^2}{k_1}\gamma_1$. We conclude that no solution exists if
\[\frac{a_0^2}{k_0}<-\frac{b}{\gamma_1}\le\frac{a_1^2}{k_1}\]
while in the case
\[\frac{a_1^2}{k_1}<-\frac{b}{\gamma_1}\le \frac{a_0^2}{k_0}\]
there are solutions of both types $0,1$, and the uniqueness fails.
\end{remark}

\begin{remark}\label{rem5}
We established the uniqueness of a solution of the prescribed form (\ref{tki}), (\ref{tkn}). In the case $u_0=u_1$ there is another ``non-physical'' solution $u=u_*(\xi)$ of problem (\ref{1}), (\ref{1}), (\ref{Neu}) such that $u_*(\xi)>u_1$ for all $\xi>0$ and therefore the phase transition corresponding to the temperature $u_1$ is absent (it happens instantly at the initial moment $t=0$), so that the total number of phase transitions reduces by one. The solution $u_*(\xi)$ of type $n-1$ is defined by the same expressions (\ref{tki}), (\ref{tkn}), where we now take $\xi_1=\xi_0=+\infty$.
By the similar reasons as for solutions $u(\xi)$ we can prove the existence and uniqueness of the solution $u_*$.

Observe that both the functions $\tilde u=u(|x|/\sqrt{t})$, $\tilde u_*=u_*(|x|/\sqrt{t})$ are solutions of (\ref{diff1}). This seems surprising because of the uniqueness statement of Theorem~\ref{thun}. But there is no contradiction here. The matter is that the functions $\beta(\tilde u)$, $\beta(\tilde u_*)$ take different initial data, namely
\[\beta(\tilde u)(0+,x)=\beta(u_0), \quad \beta(\tilde u_*)(0+,x)=\beta(u_0+)=\beta(u_0)+d_1.\]
Therefore, $P(0+,x)=\beta(\tilde u_*)(0+,x)-\beta(\tilde u)(0+,x)=d_1>0$, and the reasoning used in the proof of Theorem~\ref{thun} is not applicable.

Similar non-physical solutions can be constructed for the Dirichlet problem as well.
\end{remark}


\begin{thebibliography}{999}
\vskip4pt
\bibitem{CJ}
H.~S. Carslaw, J.~C. Jaeger, {\it Conduction of heat in solids, Second edition}, Oxford University Press, 1959.
\bibitem{Kam}
S.~L. Kamenomostskaya (Sh. Kamin), ``On Stefan's problem,'' {\it Mat. Sb. (N.S.)}, \textbf{53(95)}, no. 4, 489--514 (1961).
\bibitem{LSU}
O.~A. Ladyzhenskaya, V.~A. Solonnikov and N.~N. Ural'tseva,
{\it Linear and Quasi-Linear Equations of Parabolic Type}, AMS, Providence, 1968.
\bibitem{Pan2}
E.~Yu. Panov, ``On the structure of weak solutions of the Riemann problem for a degenerate nonlinear diffusion equation,'' {\it Contemporary Mathematics. Fundamental Directions}, \textbf{69}, no. 4, 676--684 (2023).
\bibitem{Pan1}
E.~Yu. Panov, ``Solutions of an Ill-Posed Stefan Problem,'' {\it J. Math. Sci.}, \textbf{274}, no.~4, 534--543 (2023).
\end{thebibliography}
\end{document}